\documentclass[10pt,reqno]{amsart}
\usepackage[utf8]{inputenc}
\usepackage{caption}
\usepackage{amsmath}
\usepackage{mathtools}
\usepackage{bm}
\usepackage{xcolor}
\usepackage{graphicx}
\usepackage{amssymb}
\usepackage{braket}
\usepackage{enumitem}
\usepackage{dsfont}
\usepackage{hyperref}

\usepackage{color}
\usepackage{color}

 \newtheorem{thm}{Theorem}[section]
 
 \newtheorem{lem}[thm]{Lemma}

 \newtheorem*{rem*}{Remark}
 
\newtheorem{theorem}{Theorem}

\newtheorem{lemma}{Lemma}

 \theoremstyle{definition}
  
 \numberwithin{equation}{section}

\newcommand\RR{\mathbb R} 

\newcommand\vct[2]{\left(\begin{array}{c}#1 \\ #2\end{array}\right)}

\newcommand\X{X}
\newcommand\Y{Y}

\newcommand{\dps}{\displaystyle}
\newcommand{\ii}{\infty}
\newcommand\R{{\ensuremath {\mathbb R} }}
\newcommand\C{{\ensuremath {\mathbb C} }}

\newcommand\1{{\ensuremath {\mathds 1} }}
\renewcommand\phi{\varphi}

\newcommand{\cR}{\mathcal{R}}

\newcommand{\cK}{\mathcal{K}}
\newcommand{\cE}{\mathcal{E}}

\newcommand{\cH}{\mathcal{H}}
\newcommand{\cL}{\mathcal{L}}

\newcommand{\alp}{\boldsymbol{\alpha}}

\renewcommand{\epsilon}{\varepsilon}
\newcommand\pscal[1]{{\ensuremath{\left\langle #1 \right\rangle}}}
\newcommand{\norm}[1]{ \left| \! \left| #1 \right| \! \right| }

\renewcommand{\ge}{\geqslant}
\renewcommand{\le}{\leqslant}
\renewcommand{\geq}{\geqslant}
\renewcommand{\leq}{\leqslant}

\renewcommand{\tilde}{\widetilde}

\newcommand{\nn}{\nonumber}

 \newcommand{\be}[1]{\begin{equation}\label{#1}}
\newcommand{\ee}{\end{equation}}

\def\squarebox#1{\hbox to #1{\hfill\vbox to #1{\vfill}}}

\begin{document}

\title[Uniqueness and non-degeneracy for a nuclear NLS equation]{Uniqueness and non-degeneracy for a nuclear nonlinear Schr\"odinger equation}

\author[M. Lewin]{Mathieu LEWIN}
\address{CNRS \& Laboratoire de Math\'ematiques (UMR 8088), Universit\'e de Cergy-Pontoise, F-95000 Cergy-Pontoise, France.}
\email{Mathieu.Lewin@math.cnrs.fr}

\author[S. Rota Nodari]{Simona ROTA NODARI}
\address{Laboratoire Paul Painlevé (UMR 8524), Université Lille 1 Sciences et Technologies, F-59655 Villeneuve d'Ascq, France.}
\email{Simona.Rota-Nodari@math.univ-lille1.fr}

\date{\today}

\begin{abstract}
We prove the uniqueness and non-degeneracy of positive solutions to a cubic nonlinear Schrödinger (NLS) type equation that describes nucleons. The main difficulty stems from the fact that the mass depends on the solution itself. As an application, we construct solutions to the $\sigma$--$\omega$ model, which consists of one Dirac equation coupled to two Klein-Gordon equations (one focusing and one defocusing).
\end{abstract}

\maketitle

\tableofcontents

\section{Introduction and main results}

\subsection{A nonlinear Schrödinger type equation}
The purpose of this paper is to study the uniqueness and non-degeneracy of solutions to a nonlinear Sch\"odinger-type equation, arising from the minimization of the following energy functional
\begin{equation}
\frac{1}{2}\int_{\R^3}\frac{|\bm\sigma\cdot\nabla \psi(x)|^2}{(1-|\psi(x)|^2)_+}\,dx-\frac{a}{4}\int_{\R^3}|\psi(x)|^4\,dx,
\label{eq:energy}
\end{equation}
under the mass constraint $\int_{\R^3}|\psi(x)|^2\,dx=1$. Here $x_+=\max(x,0)$ is the positive part, $\psi=(\psi_1,\psi_2)\in L^2(\R^3,\C^2)$ is a 2-spinor that describes the quantum state of a nucleon (a proton or a neutron),
$$\sigma _1=\left( \begin{matrix} 0 & 1
\\ 1 & 0 \\ \end{matrix} \right),\quad  \sigma_2=\left( \begin{matrix} 0 & -i \\
i & 0 \\  \end{matrix}\right),\quad  \sigma_3=\left( 
\begin{matrix} 1 & 0\\  0 &-1\\  \end{matrix}\right)$$
are the Pauli matrices and $\bm\sigma\cdot \nabla:=\sum_{j=1}^3\sigma_j\partial_{x_j}$.
The equation of interest is
\begin{equation}
	\label{eqNLS}
	-\bm{\sigma}\cdot\nabla\left(\frac{\bm{\sigma}\cdot\nabla\psi}{1-|\psi|^2}\right)+\frac{|\bm{\sigma}\cdot\nabla\psi|^2}{\left(1-|\psi|^2\right)^2}\psi-a|\psi|^2\psi+b\psi=0,
\end{equation}
with $b$ the Lagrange multiplier associated with the mass constraint. 

This equation can as well be written in the form of a system of two coupled Dirac-like equations
\begin{equation}\label{eqNLSsystem}
	\left\{
		\begin{aligned}
		&-i\bm{\sigma}\cdot\nabla\zeta+|\zeta|^2\psi-a|\psi|^2\psi+b\psi=0,\\
		&i\bm{\sigma}\cdot\nabla\psi+\left(1-|\psi|^2\right)\zeta=0\,.
		\end{aligned}
	\right.
\end{equation}
Indeed, the above model can formally be deduced from a relativistic model involving one Dirac particle coupled with two auxiliary classical fields (the so-called \emph{$\sigma-\omega$ model}), in a specific non-relativistic limit that will be described in detail below. In this limit, the equations for the classical fields can be solved explicitly, leading to the nonlinear system~\eqref{eqNLSsystem} and the corresponding nonlinear energy functional~\eqref{eq:energy}, expressed in terms of $\psi$ only.

The term $-(a/4)\int_{\R^3}|\psi|^4$ in~\eqref{eq:energy} is the usual nonlinear Schr\"odinger attraction which describes here the confinement of the nucleons. On the other hand, the denominator $(1-|\psi|^2)_+$ can be interpreted as a mass depending on the state $\psi$ of the nucleon, and it describes a phenomenon of saturation in the system. A high density $|\psi|^2$ generates a lower mass, which itself prevents from having a too high density. Mathematically speaking, this term enforces the additional constraint $0\leq|\psi|\leq 1$, which is very important for the stability of the energy~\eqref{eq:energy}. Without the $\psi$-dependent mass, the model is of course unstable and the energy functional is unbounded from below. The mass term $(1-|\psi|^2)_+$ allows us to consider the minimization of the energy~\eqref{eq:energy} in space dimensions $d\geq1$ without any limitation on $d$ and $a>0$, even if $d=3$ is the interesting physical case. We remark that the upper bound $1$ on the particle density $|\psi(x)|^2$ arises after an appropriate choice of units.

Let us emphasize that, in the model presented above, spin is taken into account since $\psi$ takes values in $\C^2$. Under the additional assumption that the state of the nucleon is an eigenfunction of the spin operator, the energy must be restricted to functions of the special form
\begin{equation}
\psi(x)=\phi(x)\begin{pmatrix}1\\ 0\end{pmatrix},
\label{eqsolspin}
\end{equation}
leading to the simpler functional
\begin{equation}
\cE_a(\phi):=\frac{1}{2}\int_{\R^d}\frac{|\nabla \phi(x)|^2}{(1-|\phi(x)|^2)_+}\,dx-\frac{a}{4}\int_{\R^d}|\phi(x)|^4\,dx.
\label{eq:energy_nospin}
\end{equation}
It is an open problem to show that minimizers of the original energy~\eqref{eq:energy} are necessarily of the special form~\eqref{eqsolspin}. In principle, the spin symmetry could be broken. In this paper we will however restrict ourselves to the simplified functional~\eqref{eq:energy_nospin}, which we study in any space dimension $d\geq1$. The corresponding Euler-Lagrange equation simplifies to 
\begin{equation}
-\nabla\cdot\left(\frac{\nabla\phi}{1-|\phi|^2}\right)+\frac{|\nabla\phi|^2}{\left(1-|\phi|^2\right)^2}\phi-a|\phi|^2\phi+b\phi=0
\label{eq:nonlinear_nospin}
\end{equation}

To our knowledge, the above model was mathematically studied for the first time in \cite{EstRot-12}, where Esteban and the second author of this paper formally derived the equation~\eqref{eqNLS} from its relativistic counterpart, and then proved the existence of radial square integrable solutions of (\ref{eq:nonlinear_nospin}).
This result has then been generalized in \cite{TreRot-13}, where the existence of infinitely many square-integrable excited states (solutions with an arbitrary but finite number of sign changes) was shown.

In \cite{EstRot-13}, Esteban and the second author used a variational approach to prove the existence of minimizers for the spin energy~\eqref{eq:energy}, for a large range of values for the parameter $a$. The model is translation-invariant, hence uniqueness cannot hold. Usual symmetrization techniques do not obviously apply due to the presence of the Pauli matrices $\sigma_k$'s but a natural conjecture is that all minimizers are of the form
\begin{equation}
\psi(x)=\phi(|x|)\begin{pmatrix}1\\ 0\end{pmatrix},
\label{eqsolrad}
\end{equation}
after a suitable space translation and a choice of spin orientation.

The approach of~\cite{EstRot-13} applies as well to the simplified no-spin model~\eqref{eq:energy_nospin}, and the proof works in any dimension. The result in this case is the following.

\begin{theorem}[Existence of minimizers in the no-spin case~\cite{EstRot-13}]\label{thm:minimizers}
Let $d\geq1$ and
\begin{equation}
E(a):=\inf\left\{\cE_a(\phi)\ :\ \int_{\R^d}\frac{|\nabla\phi|^2}{(1-|\phi|^2)_+}<\ii,\ \int_{\R^d}|\phi|^2=1\right\}.
\label{eq:E(a)}
\end{equation}
There exists a universal number $0\leq a_d<\ii$ such that 

\smallskip

\noindent $\bullet$ For $a\leq a_d$, $E(a)=0$ and there is no minimizer;

\smallskip

\noindent $\bullet$ For $a>a_d$, $E(a)<0$ and all the minimizing sequences are precompact in $H^1(\R^d)$, up to translations. There is at least one minimizer $\phi$ for the minimization problem $E(a)$ and it can be chosen such that $0\leq\phi\leq1$, after multiplication by an appropriate  phase factor. It solves the nonlinear equation~\eqref{eq:nonlinear_nospin} for some $b>0$.
\end{theorem}

The method used in~\cite{EstRot-13} to prove Theorem~\ref{thm:minimizers} is based on Lions' concentration-compactness technique~\cite{Lions-84,Lions-84b} and the main difficulty was to deal with the denominator $(1-|\phi|^2)_+$, for which special localization functions had to be introduced. Because the energy~\eqref{eq:energy_nospin} depends linearly on the parameter $a$, the function $a\mapsto E(a)$ is concave non-increasing, which is another important fact used in the proof of~\cite{EstRot-13}. 

The critical strength $a_d$ of the nonlinear attraction is the largest for which $E(a)=0$ and it can simply be defined by
$$a_d=\inf_{\substack{\phi\in H^1(\R^d)\\ 0\leq|\phi|\leq1}}\left\{\frac{\dps 2\left(\int_{\R^d}|\phi|^2\right)^{\tfrac{2}{d}}\left(\int_{\R^d}\frac{|\nabla\phi|^2}{(1-|\phi|^2)_+}\right)}{\dps\int_{\R^d}|\phi|^4}\right\}.$$
It can easily be verified that $a_1=0$ in dimension $d=1$, that 
$$a_2=\inf_{\substack{\phi\in H^1(\R^2)\\ 0\leq|\phi|\leq1}}\left\{2\frac{\dps\norm{\phi}_{L^2(\R^2)}^2\norm{\nabla\phi}_{L^2(\R^2)}^2}{\norm{\phi}_{L^4(\R^2)}^4}\right\}>0$$
is related to the Gagliardo-Nirenberg-Sobolev constant in dimension $d=2$, and that $a_d>0$ in higher dimensions. Estimates on $a_d$ have been provided in dimension $d=3$ in~\cite{EstRot-13} and similar bounds can be derived in higher dimensions by following the same method.

\subsection{Uniqueness and non-degeneracy of solutions}

After the two works~\cite{EstRot-12,EstRot-13}, it remained an open problem to show that minimizers are all radial and unique, up to a possible translation and multiplication by a phase factor. The purpose of this paper is to answer this question. Our main result is the following.

\begin{theorem}[Uniqueness and non-degeneracy in the no-spin case]\label{thm:main}
The nonlinear equation~\eqref{eq:nonlinear_nospin} has no non-trivial solution $0<\varphi<1$ in $L^2(\RR^d)$ when $0< a\leq 2b$.
For $a>2b>0$, the nonlinear equation~\eqref{eq:nonlinear_nospin} admits a unique solution $0<\phi<1$ that tends to 0 at infinity, modulo translations and multiplication by a phase factor. It is radial, decreasing, and non-degenerate.
\end{theorem}

This theorem is the equivalent of a celebrated similar result for the nonlinear Schr\"odinger equation (see, e.g.,~\cite[App. B]{Tao-06} and~\cite{Frank-13} for references). Our main contribution is the remark that the equation~\eqref{eq:nonlinear_nospin} can be rewritten in terms of $u:=\arcsin(\phi)$ as a simpler nonlinear Schr\"odinger equation
\begin{equation}
-\Delta u+b\sin(u)\cos(u)-a\sin^3(u)\cos(u)=0.
\label{eq:simpler_intro}
\end{equation}
Applying a classical argument of McLeod~\cite{McLeod-93} (as explained in~\cite[App. B]{Tao-06} and in~\cite{Frank-13}) allows to prove the non-degeneracy and uniqueness in the radial case. That any solution of~\eqref{eq:nonlinear_nospin} is necessarily radial decreasing then follows from the moving plane method~\cite{GidNiNir-81,LiNi-93}. The proof of Theorem~\ref{thm:main} is provided in Section~\ref{sec:proof} below.

Let us remark that, since equation~\eqref{eq:nonlinear_nospin} is invariant under multiplications by a phase factor, we can always suppose that a solution $\varphi$ is real-valued. Hence, in~\cite[Appendix A.1]{EstRot-13} it has been proved that any solution $\varphi\in H^1(\RR^d)$ is such that $|\varphi|^2\le 1$ a.e. in $\RR^d$ whenever $a\ge b>0$. As a consequence, the change of variables $u=\arcsin(\varphi)$ makes sense whenever $a\ge b>0$. 

\subsection{Application: solutions to a Dirac Klein-Gordon equation}

As an application of Theorem~\ref{thm:main}, we are able to construct a branch of solutions of the underlying Dirac equation, that converges to the non-relativistic solution $\phi$ in the limit, thereby justifying the formal arguments of~\cite{EstRot-12}. We explain this now.

We restrict ourselves to $d=3$ for simplicity (but the results are similar in other dimensions). We consider one relativistic nucleon in interaction with two scalar fields $S$ (the \emph{$\sigma$--field}) and $V$ (the \emph{$\omega$--field}). As described for instance in~\cite{Walecka-74,SerWal-86,Reinhard-89,Ring-96,Walecka-04,MenTokZhoZhaLonGen-06}, the corresponding equation is
\begin{equation}
\begin{cases}
-i\alp\cdot\nabla \Psi+\beta(m+S)\Psi+V\Psi=(m-\mu)\Psi,\\
(-\Delta+m_\sigma^2)S=-g_\sigma^2\Psi^*\beta\Psi,\\
(-\Delta+m_\omega^2)V=g_\omega^2|\Psi|^2,\\
\end{cases}
\label{eq:Dirac}
\end{equation}
where
$$\alpha_k=\begin{pmatrix}
0&\sigma_k\\ \sigma_k & 0
\end{pmatrix},\qquad k=1,2,3,\ 
\beta=\begin{pmatrix}
\1_2 & 0\\
0 & -\1_2
\end{pmatrix}
$$
are the Dirac matrices and $\Psi\in L^2(\R^3,\C^4)$ is now a 4-spinor. The wavefunction $\Psi$ should in principle be normalized in $L^2$ but, here, we think of fixing $\mu$ instead of imposing $\norm{\Psi}_{L^2}=1$. Any non-trivial solution $\Psi$ to~\eqref{eq:Dirac} also gives a normalized solution after an appropriate change of parameters. In most physics papers, the equation for the $\sigma$-field $S$ contains a nonlinear term as well (for instance including vacuum polarization effects~\cite{Reinhard-89}), 
$$\big(-\Delta+m_\sigma^2+U'(S)\big)S=-g_\sigma^2\Psi^*\beta\Psi,$$
but we restrict ourselves to the simpler linear case for convenience.

The fields $S$ and $V$ are respectively focusing and defocusing, which can be seen from the different signs in the two Klein-Gordon equations. On the other hand, they have very different effects, since $S$ modifies the mass $m$ in the same way for the upper and lower spinors, whereas $V$ is repulsive for the upper spinor and attractive for the lower spinor. This statement is clarified when the Dirac equation is written in terms of
$$\Psi=\vct{\psi}{\zeta}$$
as
\begin{equation}
\label{eqnrlphichi3dintro}
\begin{cases}
-i\bm{\sigma}\cdot\nabla \zeta +(S+V+\mu)\psi=0,\\
-i\bm{\sigma}\cdot\nabla \psi =(2m-\mu+S-V)\zeta,\\
(-\Delta+m_\sigma^2)S=-g_\sigma^2(|\psi|^2-|\zeta|^2),\\
(-\Delta+m_\omega^2)V=g_\omega^2(|\psi|^2+|\zeta|^2).\\
\end{cases}
\end{equation}
We see that $S+V$ and $S-V$ respectively appear in the two equations.

In our units, the non-relativistic limit corresponds to $m,m_\sigma,m_\omega\to\ii$, with all the masses being of the same order. On the contrary to atomic physics, in nuclear physics the coupling constants $g_\omega$ and $g_\sigma$ are very large, comparable to the masses. It is therefore customary to work in a regime where $g_\omega/m_\omega$ and $g_\sigma/m_\sigma$ are fixed or, even, large. In the two Klein Gordon equations, the Laplacian can then be neglected in such a way that
$$S\simeq -\frac{g_\sigma^2}{m_\sigma^2}(|\psi|^2-|\zeta|^2)\quad\text{and}\quad V\simeq \frac{g_\omega^2}{m_\omega^2}(|\psi|^2+|\zeta|^2)$$
and hence
$$S+V\simeq \left(\frac{g_\omega^2}{m_\omega^2}-\frac{g_\sigma^2}{m_\sigma^2}\right)|\psi|^2+\left(\frac{g_\omega^2}{m_\omega^2}+\frac{g_\sigma^2}{m_\sigma^2}\right)|\zeta|^2,$$
$$S-V\simeq -\left(\frac{g_\omega^2}{m_\omega^2}+\frac{g_\sigma^2}{m_\sigma^2}\right)|\psi|^2-\left(\frac{g_\omega^2}{m_\omega^2}-\frac{g_\sigma^2}{m_\sigma^2}\right)|\zeta|^2.$$
As usual, in the non-relativistic regime, the lower spinor $\chi$ is of order $1/\sqrt{m}$. Simple effective equations will then be obtained in the limit.


\subsubsection*{The $\sigma$ model}
In order to better illustrate the regime of interest for the $\sigma$--$\omega$ model, let us first discuss the case of the $\sigma$ model, in which $V\equiv0$ and $g_\omega\equiv0$. The equation~\eqref{eqnrlphichi3dintro} then reduces to
\begin{equation}
\label{eq:sigma}
\begin{cases}
-i\bm{\sigma}\cdot\nabla \zeta +S\psi+\mu\psi=0,\\
-i\bm{\sigma}\cdot\nabla \psi =(2m-\mu+S)\zeta,\\
(-\Delta+m_\sigma^2)S=-g_\sigma^2(|\psi|^2-|\zeta|^2),\\
\end{cases}
\end{equation}
The interesting regime is then $g_\sigma/m_\sigma$ of order 1, say $(g_\sigma/m_\sigma)^2=\kappa$ fixed. It can be proved that $2m-\mu+S\simeq 2m$ and the usual NLS equation is recovered in the limit, after a simple scaling. The precise result is the following.

\begin{theorem}[Non-relativistic limit of the $\sigma$ model]\label{thm:sigma}
Let $\kappa,\mu,c$ be positive constants. Then for $m$ large enough, the equation~\eqref{eq:sigma} admits a branch of solutions of the special form
\begin{equation}\label{eq:form_Dirac_sigma}
	\Psi_m(x)=\left(
	\begin{array}{c}
		\varphi_m(|x|)\vct{1}{0}\\
		-i\chi_m(|x|)\;\bm\sigma\cdot\frac{x}{|x|}\vct{1}{0}
	\end{array}
	\right),
\end{equation}
with
\begin{equation}
m_\sigma=cm,\qquad  \left(\frac{g_\sigma}{m_{\sigma}}\right)^2=\kappa.
\label{eq:NR-limit-parameters_sigma}
\end{equation}
In the limit $m\to\ii$, we have 
$$\varphi_m\big(\cdot/\sqrt{m}\big)\to \phi_{\rm NLS}\quad\text{and}\quad 2\sqrt{m}\,\chi_m\big(\cdot/\sqrt{m}\big)\to \phi'_{\rm NLS}$$
strongly in $H^2(\R^3)$, where $\phi_{\rm NLS}$ is the unique positive radial solution of
\begin{equation}
-\Delta\phi_{\rm NLS}-2\kappa\phi^3_{\rm NLS}+2\mu\phi_{\rm NLS}=0.
\label{eq:NLS}
\end{equation}
\end{theorem}

Functions of the form~\eqref{eq:form_Dirac_sigma} have the lowest possible total angular momentum~\cite[Sec.~4.6.4]{Thaller}. The theorem can be shown by following step by step the method of Section~\ref{sec:relativistic_limit}, using the non-degeneracy of the NLS ground state. Its proof will not be provided in this paper for shortness.

Theorem~\ref{thm:sigma} is not satisfactory from a physical point of view. Indeed, the limit $\phi_{\rm NLS}$ is considered physically unstable since the corresponding energy functional is unbounded from below in dimension 3. Furthermore, in practice $\kappa$ is very large and the corresponding $\phi_{\rm NLS}$ is then very peaked at the origin. In real nuclei, many forces are in action and they tend to compensate in order to avoid this collapse at $0$. It is therefore important to take the $\omega$ field into account.


\subsubsection*{The $\sigma$--$\omega$ model}
For the $\sigma$--$\omega$ model, the interesting regime is when the parameters $g_\sigma^2/m_\sigma^2$ and $g_\omega^2/m_\omega^2$ behave like $m$, whereas $g_\sigma^2/m_\sigma^2-g_\omega^2/m_\omega^2$ stays bounded, which is the cancellation between the two scalar fields mentioned before. Even if $g_\sigma^2/m_\sigma^2$ diverges, the model still has a nice bounded limit $\phi$, which is precisely the non-relativistic ground state studied in the previous section. 

\begin{theorem}[Non-relativistic limit of the $\sigma$--$\omega$ model]\label{thm:Dirac}
Let $\theta,\lambda,\mu,C,D$ be positive constants such that $\lambda>2\mu\theta$. Then for $m$ large enough, the equation~\eqref{eq:Dirac} admits a branch of solutions of the special form
\begin{equation}
	\Psi_m(x)=\left(
	\begin{array}{c}
		\varphi_m(|x|)\vct{1}{0}\\
		-i\chi_m(|x|)\;\bm\sigma\cdot\frac{x}{|x|}\vct{1}{0}
	\end{array}
	\right),
	\label{eq:form_Dirac}
\end{equation}
with
\begin{equation}
m_\sigma^2=Cm^2,\quad m^2_\omega-m_\sigma^2=D,\quad \left(\frac{g_\sigma}{m_{\sigma}}\right)^2=\theta m,\quad \left(\frac{g_\sigma}{m_{\sigma}}\right)^2-\left(\frac{g_\omega}{m_{\omega}}\right)^2=\lambda.
\label{eq:NR-limit-parameters}
\end{equation}
In the limit $m\to\ii$, we have 
$$\sqrt{\theta}\,\varphi_m\big(\cdot/\sqrt{m}\big)\to \phi\quad\text{and}\quad 2\sqrt{\theta m}\,\chi_m\big(\cdot/\sqrt{m}\big)\to \frac{\phi'}{1-\phi^2}$$
strongly in $H^2(\R^3)$, where $\phi$ is the unique positive solution of~\eqref{eq:nonlinear_nospin} with $a=2\lambda/\theta$ and $b=2\mu$.
\end{theorem}

We refer to~\cite{Walecka-74,Walecka-04},~\cite[Sec.~3]{Reinhard-89} and~\cite[Sec.~2.3]{Ring-96} for a discussion of the validity of this regime for standard nucleons. Typical physical values for the parameters of the model are provided in~\cite[Table~3.1]{Ring-96}.

The proof of Theorem~\ref{thm:Dirac} is provided in Section~\ref{sec:relativistic_limit} and it is based on the implicit function theorem. In other words, we see~\eqref{eqnrlphichi3dintro} as a small  perturbation of~\eqref{eqNLSsystem} and we use the non-degeneracy of $\phi$ to construct a solution. Remark that, thanks to the non-degeneracy property proved in Section~\ref{sec:nondegeneracy_L2}, this argument gives also the local uniqueness of the solution to~\eqref{eqnrlphichi3dintro} around $\varphi$, modulo translations and multiplication by a phase factor. The exact same reasoning can be used for proving Theorem~\ref{thm:sigma}. A similar argument has for instance been used in~\cite{Lenzmann-09}. 

We hope that our work will stimulate further research on this model.

\medskip

\noindent\textbf{Acknowledgement.} The authors acknowledge financial support from the European Research Council (FP7/2007-2013 Grant Agreement MNIQS 258023) and the ANR (NoNAP 10-0101) of the French Ministry of Research. Moreover, the research of the second author was supported by the Labex CEMPI (ANR-11-LABX-0007-01).

\section{Proof of Theorem~\ref{thm:main}}\label{sec:proof}

This section is devoted to the proof of Theorem~\ref{thm:main}, which is split in several steps. In the next section, we explicit the change of variable $u=\arcsin(\phi)$  and combine it with symmetric rearrangement to deduce that the minimization problem $E(a)$ can be restricted to radial decreasing functions. This step is not necessary for our analysis but we mention it for completeness, as it gives a simpler existence proof than in~\cite{EstRot-12}. Then, in Section~\ref{sec:moving_planes}, we use the moving plane method to conclude that positive solutions to the nonlinear equation~\eqref{eq:nonlinear_nospin} are radial decreasing. Section~\ref{sec:uniqueness_radial} is devoted to the uniqueness of radial solutions. Finally, we prove the non-degeneracy of the linearized operator in the whole of $L^2(\R^d)$ (modulo the trivial symmetries of the problem) in Section~\ref{sec:nondegeneracy_L2}.

\subsection{Minimizers are radial decreasing}\label{sec:rearrangement}

We recall that the energy functional is
\begin{equation}
\cE_a(\phi):=\frac{1}{2}\int_{\R^d}\frac{|\nabla \phi(x)|^2}{(1-|\phi(x)|^2)_+}\,dx-\frac{a}{4}\int_{\R^d}|\phi(x)|^4\,dx
\label{eq:energy_nospin2}
\end{equation}
which we study on the subset of $\phi$'s in $L^2(\R^d)$ such that $\int_{\R^d}|\phi|^2=1$ and
$$\int_{\mathbb{R}^d}\frac{|\nabla \phi|^2}{(1-|\phi|^2)_+}<\ii.$$
Since $(1-|\phi|^2)_+\leq 1$, it is clear that any such $\phi$ must be in $H^1(\R^d)$. It was proved in~\cite[Lem. 2.1]{EstRot-13} that it must also satisfy $0\leq|\phi|\leq1$ a.e. The nonlinear term $\int_{\R^d}|\phi|^4$ is then well defined and, since $0\leq|\phi|\leq1$, we conclude that $E(a)\geq -a/4$.

By using rearrangement inequalities and the change of variable $u=\arcsin(\phi)$, we are able to prove that minimizers are always radial-decreasing. This can be used to simplify the proof of Theorem~\ref{thm:minimizers} of~\cite{EstRot-12}.

\begin{lemma}[Minimizers are radial decreasing]\label{lem:rearrangement}
For every $a\geq0$, the minimization problem $E(a)$ can be restricted to radial non-increasing functions. Furthermore, any minimizer of $E(a)$, when it exists, is positive and radial-decreasing, after a possible translation and multiplication by a phase factor.
\end{lemma}

\begin{proof}
First we recall that $|\nabla\phi|^2\geq|\nabla|\phi||^2$ a.e., see~\cite[Thm. 7.8]{LieLos-01}. Hence $\cE_a(\phi)\geq\cE_a(|\phi|)$ and the minimization problem can be restricted to functions satisfying $0\leq\phi\leq1$, which we assume from now on. Let then $\phi^*$ be the Schwarz rearrangement of $\phi$. Using that 
$$\nabla\arcsin(\phi)=\frac{\nabla \phi}{\sqrt{1-\phi^2}},$$
we see that
$$\cE_a(\phi)=\frac12\int_{\R^d}|\nabla\arcsin(\phi)|^2-\frac{a}{4}\int_{\R^d}|\phi|^4.$$
Next, we have $\int_{\R^d}|\nabla u|^2\geq \int_{\R^d}|\nabla u^*|^2$ for all $u\in H^1(\R^d)$ and, since $\arcsin$ is increasing, $\arcsin(\phi)^*=\arcsin(\phi^*)$, by~\cite[Chap. 3 \& Lem. 7.17]{LieLos-01}. We conclude that $\cE_a(\phi)\geq\cE_a(\phi^*)$ and the minimization can be restricted to radial non-decreasing functions. 

If $\phi$ is a (possibly non-symmetric) minimizer with $0\leq\phi\leq1$, then $\phi^*$ is also a minimizer and we have $\int_{\R^d}|\nabla\arcsin\phi|^2=\int_{\R^d}|\nabla\arcsin\phi^*|^2$. In general, this does not imply that $\phi$ is itself radial-decreasing, but this will be proved using the nonlinear equation. Denoting $u=\arcsin(\phi)$ and $u^*=\arcsin(\phi^*)$, we see that $u\geq0$ must solve the Euler-Lagrange equation
$$-\Delta u+\frac{a}{2}\sin(2u)\left(\frac{b}{a}-\sin^2(u)\right)=0.
$$
In particular, $u$ must be the first eigenvector of the Schr\"odinger operator 
$$-\Delta+\frac{a\sin(2u)}{2u}\left(\frac{b}{a}-\sin^2(u)\right)$$
and therefore $u>0$. The real-analyticity of $u$ (see, e.g.,~\cite{Morrey-58}) combined with the equality $\int_{\R^d}|\nabla u|^2=\int_{\R^d}|\nabla u^*|^2$ now implies that $u=u^*$, hence $\phi=\phi^*$, after an appropriate space translation, by~\cite{BroZie-88,FerVol-03}. Finally, if $\phi$ is an arbitrary minimizer, the equality $|\nabla\phi|^2=|\nabla|\phi||^2$ implies $\phi=e^{i\theta}|\phi|$ by~\cite[Thm. 7.8]{LieLos-01}. This concludes the proof of the lemma.
\end{proof}

\subsection{Positive solutions are radial decreasing}\label{sec:moving_planes}

In the previous section, we have shown using rearrangement inequalities that minimizers of $\cE_a$ are necessarily radial-decreasing. Here we use the moving plane method to prove that non-negative solutions of the equation~\eqref{eq:nonlinear_nospin} are also all radial decreasing, which of course also implies Lemma~\ref{lem:rearrangement}.

We recall that the nonlinear equation~\eqref{eq:nonlinear_nospin} can be rewritten in terms of $u=\arcsin(\phi)$ as
\begin{equation}
-\Delta u+\frac{a}{2}\sin(2u)\left(\frac{b}{a}-\sin^2(u)\right)=0.
\label{eq:NLS_u}
\end{equation}
We also remark that $u\in H^1(\R^d)$ when $\phi\in H^1(\R^d)$ and $\int_{\R^d}|\nabla\phi|^2(1-|\phi|^2)_+^{-1}<\ii$.
For simplicity of notation, we denote
\begin{equation}
F(u):=\frac{a}{2}\sin(2u)\left(\sin^2(u)-\frac{b}{a}\right).
\label{eq:def_F}
\end{equation}

\begin{lemma}[Positive solutions are radial-decreasing]
Let $a,b>0$ and $u\in L^2(\R^d)$ be a non-trivial solution of~\eqref{eq:NLS_u} with $0< u\leq\pi/2$. Then, $u$ is radial decreasing about some point in $\R^d$.
\end{lemma}

\begin{proof}
Elliptic regularity gives that $u\to0$ at infinity. Then the result follows immediately from the famous moving plane method. Indeed, noticing that $F(0)=F''(0)=0$ and $F'(0)=-b<0$, we may use~\cite[Thm. 2]{GidNiNir-81}.
\end{proof}

We have proved that any solution to the equation~\eqref{eq:NLS_u} must be radial-decreasing. The next step consists in studying the uniqueness of radial solutions.

\subsection{Uniqueness and non-degeneracy in the radial case}\label{sec:uniqueness_radial}

In this section, we study radial solutions to the equation~\eqref{eq:NLS_u}, which then solve
\begin{equation}
\label{eqNLSsinusradial}
\begin{cases}
\dps u''+\frac{d-1}{r}u'+\frac{a}{2}\sin(2u)\left(\sin^2(u)-\frac{b}{a}\right)=0\quad\text{on $\R_+$}\\
u'(0)=0
\end{cases}
\end{equation}
and we concentrate on showing the uniqueness of positive solutions such that $(u(r),u'(r))\to0$ when $r\to\ii$. In dimensions $d\geq2$, the condition $u'(0)=0$ is necessary to avoid a singularity at the origin. In dimension $d=1$, the solution is known to be even about one point and, after a suitable translation we may always assume $u'(0)=0$ as well. More precisely, to prove the existence of solutions in dimension $d=1$, we use the fact that in this case the local energy
\begin{equation}
H(r)=\frac{u'(r)^2}{2}+a\frac{\sin^4(u(r))}{4}-b\frac{\sin^2(u(r))}{2}
\label{eq:energy_Hamiltonian}
\end{equation}
is conserved along the trajectories. However, in dimension $d\ge 2$, the energy $H$ defined by \eqref{eq:energy_Hamiltonian}, decreases:
$$H'(r)=-\frac{(d-1)}{r}u'(r)^2.$$    

The solutions $u_y$ to~\eqref{eqNLSsinusradial} are parametrized by $u_y(0):=y\in(0,\pi/2)$. Using the same arguments as in the proof of~\cite[Lem.~2.6]{EstRot-12} and in particular the fact that the energy $H$ is non-increasing, we can easily show that a solution starting at $y\geq\pi/2$ stays  bigger than $\pi/2$ and hence cannot tend to $0$ at infinity. Moreover, note that the equation~\eqref{eqNLSsinusradial} has the three stationary solutions $u\equiv0$, $u\equiv\pi/2$ and $u\equiv\arcsin(\sqrt{b/a})$. Hence $u(0)\notin\{0,\arcsin(\sqrt{b/a}),\pi/2\}$ is necessary. The following is a reformulation of the result of~\cite{EstRot-12} that was expressed in terms of $\phi=\sin(u)$.

\begin{theorem}[Existence of solutions~\cite{EstRot-12}]\label{thm:shooting}
For $0<a\leq 2b$, there is no non-trivial solution $u$ to~\eqref{eqNLSsinusradial}, such that $u\to0$ at infinity.

For $a>2b>0$, there exists one positive solution $Q$ to~\eqref{eqNLSsinusradial}, such that $(Q,Q')\to(0,0)$ at infinity. It is decreasing, starts at 
\begin{align*}
	&Q(0)=\bar y=\arcsin(\sqrt{2b/a}) & \text{for}\ d=1,\\
	&Q(0)=\bar y\in \big(\arcsin(\sqrt{2b/a}),\pi/2\big)& \text{for}\  d\ge 2,
\end{align*}
and has the following behavior at infinity:
\begin{equation}
Q(r)\underset{r\to\ii}\sim C\,\frac{e^{-\sqrt{b}r}}{r^{\tfrac{d-1}{2}}}\qquad Q'(r)\underset{r\to\ii}\sim -\sqrt{b}\,C\,\frac{e^{-\sqrt{b}r}}{r^{\tfrac{d-1}{2}}},
\label{eq:exponential_decay}
\end{equation}
for some $C>0$.
\end{theorem}

\begin{figure}[p]\centering
\includegraphics[width=9cm]{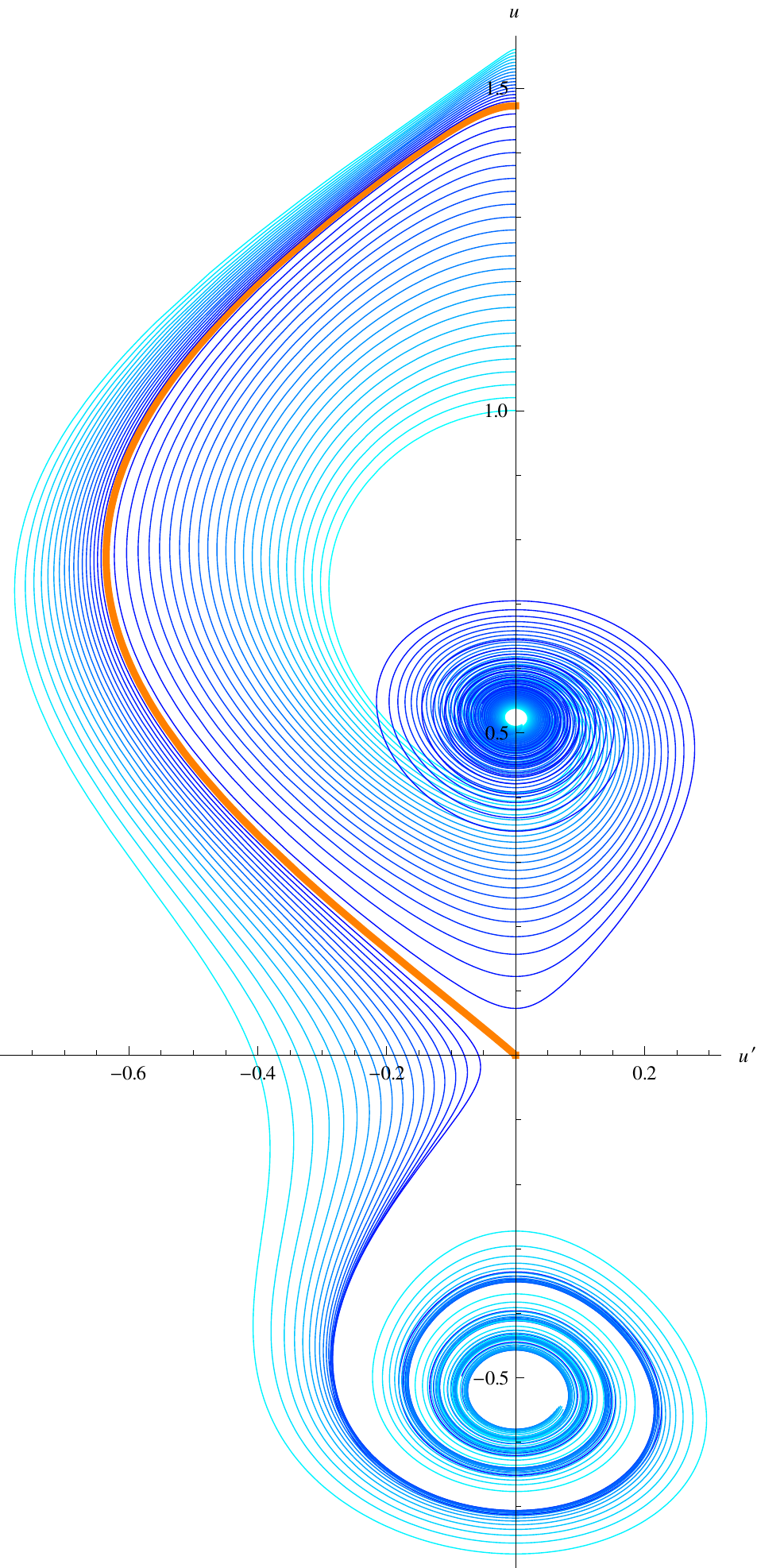}
\caption{Phase portrait with several solutions $(u'_y,u_y)$ including the ground state $Q$, for $a=4$ and $b=1$, in dimension $d=3$.\label{fig:portrait}}
\end{figure}

The proof used in~\cite{EstRot-12}, which is presented for $d=3$ but can be generalized for all $d\ge 2$, is based on a shooting method consisting in increasing $y$ continuously starting from $0$ (Figure~\ref{fig:portrait}). A byproduct of the proof is that all the other solutions $u_y$ with $0< y=u_y(0)<\bar y$ do not tend to $0$ at infinity. Hence it will remain to prove that the solutions $u_y$ with $\bar y<y<\pi/2$ necessarily vanish at some point $r_y\in\R_+$. The explicit decay rate~\eqref{eq:exponential_decay} was not stated in~\cite{EstRot-12}, but it is a classical fact whose proof can for instance be read in~\cite{GidNiNir-81}. As remarked above, the result of Lemma~\ref{lem:rearrangement} can be use to simplify this proof.

For completeness we quickly explain the non-existence part in Theorem~\ref{thm:shooting} which is needed below and is itself taken from~\cite[Prop 2.1]{EstRot-12}. The idea is to use that the local energy \eqref{eq:energy_Hamiltonian} is non-increasing.
This implies that any solution satisfying $u'(0)=0$ and $(u,u')\to(0,0)$ at infinity must be such that 
$$1>\sin^2(u(0))\geq \frac{2b}{a}.$$
Hence $a/2b>1$ is a necessary condition for the existence of $u$. Moreover, we see that $\bar y>\arcsin(\sqrt{2b/a})$ which is strictly above the stationary solution $\arcsin(\sqrt{b/a})$.

The main result of this section is the following

\begin{theorem}[Uniqueness and non-degeneracy of radial ground states]\label{thm:radial}
For $a>2b>0$ and $d\geq1$, the solution $Q$ of Theorem~\ref{thm:shooting} is the only non-trivial positive solution $u$ of~\eqref{eqNLSsinusradial} such that $(u,u')\to(0,0)$ at infinity.

Furthermore, $Q$ is non-degenerate: the unique solution $v$ to
\begin{equation}
\begin{cases}
\dps L(v)=v''+\frac{d-1}{r}v'+F'(Q)v=0\\
v(0)=1\\
v'(0)=0
\end{cases}
\label{eq:linearized}
\end{equation}
diverges exponentially fast when $r\to\ii$. More precisely, if $d\ge 2$, $v$ satisfies $v(r)\to-\ii$ and $v'(r)\to-\ii$ exponentially fast when $r\to\ii$.
\end{theorem}

\begin{proof}
In dimension $d=1$, the result follows immediately from the Hamiltonian feature of the problem, based on the energy~\eqref{eq:energy_Hamiltonian} and the fact that $(0,0)$ is a non-degenerate critical point of $H$, when $b>0$. In particular, the divergence of the solution $v$ to the linearized equation~\eqref{eq:linearized} can be proved by computing the Wronskian $(v'Q'-vQ'')'=Q'L(v)-vL(Q')=0$, using that $L(Q')=0$. We deduce that $v'(r)Q'(r)-v(r)Q''(r)=-Q''(0)=a/2\sin(2Q(0))(\sin^2Q(0)-b/a)>0$ which cannot converge to $0$ at infinity. 

In the following we assume $d\geq2$. There are many existing results dealing with the uniqueness (and, often, the non-degeneracy as well) of radial solutions to semi-linear equations of the type $\Delta u+F(u)=0$. After the pioneering works on the NLS nonlinearity~\cite{Coffman-72,Kwong-89}, many authors introduced various conditions on the function $F$ that ensure uniqueness, see, e.g.~\cite{PelSer-83,LeoSer-87, KwoZha-91,McLeod-93,SerTan-00}. Our particular function $F$ as defined in~\eqref{eq:def_F} satisfies some of the assumptions required in these works. For instance uniqueness can be directly obtained from~\cite[Thm. 1']{SerTan-00} in dimensions $d\geq3$, by means of Lemma~\ref{lem:prop_F} below. On the other hand, the non-degeneracy is sometimes not explicitly stated in those works, although often shown in the middle of the proof. For clarity, we will therefore quickly explain the proof of the theorem, following the approach of McLeod in~\cite{McLeod-93} and its summary in~\cite[App. B]{Tao-06} and~\cite{Frank-13}. 

The main properties of the function $F$ that make everything works are summarized in the following

\begin{lemma}[Elementary properties of $F$]\label{lem:prop_F}
Let $F$ be defined as in~\eqref{eq:def_F} on $\left(0,\frac{\pi}{2}\right)$, with $a>2b>0$. Then
\begin{enumerate}
\item $F$ is negative on $(0,\arcsin(\sqrt{b/a}))$ and positive on $(\arcsin(\sqrt{b/a}),\pi/2)$ with $F'(\arcsin(\sqrt{b/a}))>0$;
\item $x\mapsto xF'(x)/F(x)$ is decreasing on $(\arcsin(\sqrt{b/a}),\pi/2)$;
\item for every $\lambda>1$, the function 
\begin{equation}
I(x):=xF'(x)-\lambda F(x)
\label{eq:def_I}
\end{equation}
has exactly one root $x_*\in(\arcsin(\sqrt{b/a}),\pi/2)$, at which we have $I'(x_*)<0$.
\end{enumerate}
\end{lemma}

The above properties of $F$ are somehow inherited from the NLS case, since $F(x)=\cos(x)P(\sin(x))$ with $P(\xi)=a\xi^3-b\xi$. Below we will 
not use the property \textit{(2)}, but rather \textit{(3)} (which itself follows from \textit{(2)}). We however state \textit{(2)} since the monotonicity of $xF'(x)/F(x)$ appears in many works, including for instance~\cite{KwoZha-91} and~\cite{SerTan-00}. The proof of Lemma~\ref{lem:prop_F} will be provided at the end of the proof of the theorem. The `$I$' function~\eqref{eq:def_I} appears as well in~\cite{McLeod-93}, where an additional assumption on the behavior of $x_*$ was required.

Now, we assume that $a>2b>0$ and we look at the solutions $u_y$ of~\eqref{eqNLSsinusradial} with $u_y(0)=y$ and $u'(0)=0$, and we let $y$ vary in $(0,\pi/2)$. 
Note that the function $(y,r)\mapsto u_y(r)$ is smooth (indeed real-analytic since $F$ is analytic). Following~\cite{McLeod-93}, we introduce the sets 
$$S_+=\big\{y\in(0,\pi/2)\ :\ \min_{\R^+} u_y>0\big\},$$
$$S_0=\big\{y\in(0,\pi/2)\ :\ u_y>0\ \text{and}\ \lim_{r\to\ii}u_y(r)=0\big\},$$
$$S_-=\big\{y\in(0,\pi/2)\ :\ u_y(r_y)=0\ \text{for some (first) $r_y>0$}\big\},$$
which form a partition of $(0,\pi/2)$. As we have recalled above, since the energy $H$ is decreasing along a solution, we have $(0,\arcsin(\sqrt{2b/a}))\subset S_+$. This was actually shown in~\cite{EstRot-12}, where the solution $Q=u_{\bar y}$ is constructed by looking at the supremum of $S_+$. In particular, $S_0\neq\emptyset$. If $y\in S_0$ we let for convenience $r_y:=+\ii$. Since $(r,y)\mapsto u_y(r)$ is smooth, it can easily be proved that $S_-$ is open. The same holds for $S_+$, but the proof is more difficult. The idea is that the points of $S_-$ are characterized by the fact that the trajectory in phase space crosses first the horizontal axis (that is, $u_y$ vanishes before $u'_y$), whereas for $y\in S_+$ it only crosses the vertical axis ($u'_y$ vanishes and $u_y$ does not), see~Figure~\ref{fig:portrait}.

\begin{lemma}\label{lem:prop_S_-}
Let $y\in S_0\cup S_-$. Then $u_y'<0$ on $(0,r_y)$, that is, $u_y$ vanishes before $u'_y$. In particular, $u_y$ is strictly decreasing on $(0,r_y)$. 
\end{lemma}

\begin{proof}
The proof is again based on the monotonicity of the energy $H$ and it can for instance be read in~\cite[Lem. 3]{PelSer-83}. The idea is the following. We denote for simplicity $u=u_y$ and $u'=u'_y$. First, since 
$S_0\cup S_-\subset(\arcsin(\sqrt{2b/a},\pi/2)$, then we have from~\eqref{eqNLSsinusradial}
$u''(0)=-F(u(0))/d<0$
and hence $u'(r)<0$ for small $r>0$. On the other hand $u'(r_y)<0$ (since $r_y$ is the first root of $u=u_y$ and the latter cannot have double zeroes). Assuming that $u'$ changes sign before $r_y$ implies that $u$ must have a local strict minimum at some point $0<r'<r_y$, at which $u(r')>0$. Then, since $\lim_{r\to r_y}u(r)=0$, there must be another later point $r''<r_y$ at which $u(r'')=u(r')$. However, we have
$$\frac{u'(r'')^2}{2}=H(r'')-H(r')=\int_{r'}^{r''}H'(s)\,ds=-(d-1)\int_{r'}^{r''}\frac{u'(s)^2}{s}\,ds<0,$$
a contradiction.
\end{proof}

\begin{lemma}\label{lem:prop_S_+}
Let $y\in S_+$. Then $u_y'$ vanishes at least once and, for the first positive root $r'_y$ of $u_y'$, we have $H(r'_y)<0$. The set $S_+$ is open. 
\end{lemma}

\begin{proof}
The proof follows the presentation of~\cite{Frank-13} and it goes as follows.  We denote for simplicity $u=u_y$ and $u'=u'_y$. If $y=\arcsin(\sqrt{b/a})$, then $u\equiv \arcsin(\sqrt{b/a})$ and $H(r)<0$ for all $r\in \RR^+$. Hence, let $y\neq\arcsin(\sqrt{b/a})$. First we claim that $u'$ must vanish. Otherwise $u$ is decreasing whenever $y\in S_+\cap \left(\arcsin(\sqrt{b/a}), \frac{\pi}{2}\right)$ and increasing if $y\in S_+\cap \left(0,\arcsin(\sqrt{b/a})\right)$. In both cases, $u$ has a positive limit $0<\alpha<\frac{\pi}{2}$ at infinity. Using the equation, we see that $F(\alpha)=0$, hence $\alpha=\arcsin(\sqrt{b/a})$. Next, following~\cite{BerLioPel-81,Frank-13}, we look at $U:=r^{(d-1)/2}(u-\arcsin(\sqrt{b/a}))>0$ which solves the equation
$$U''=\left(\frac{(d-1)(d-3)}{4r^2}-\frac{F(u)}{u-\arcsin(\sqrt{b/a})}\right)U.$$
At infinity we have $F(u)(u-\arcsin(\sqrt{b/a}))^{-1}\to 2b(a-b)>0$, hence $U''(r)\sim_{r\to\ii} -2b(a-b)U(r)$, which easily leads to a contradiction. We conclude that $u'$ vanishes and we denote by $r'_y$ its first root. 

Next we distinguish two cases. First, if $y\leq \arcsin(\sqrt{b/a})$, then $H(0)<0$ and, by~\eqref{eq:energy_Hamiltonian}, $H(r)<0$ for all $r>0$ and in particular $H(r'_y)<0$. Second, if $y>\arcsin(\sqrt{b/a})$, then $u'$ is negative for small $r$ (due to the fact that $u''(0)=-F(y)/d<0$). Since $u''(r'_y)\neq0$ (otherwise $F(u(r'_y))=0$ and $u$ is constant), we see that $u$ must attain a local minimum at $r'_y$. From the equation~\eqref{eqNLSsinusradial}, this yields $F(u(r'_y))<0$ and hence $u(r_y')<\arcsin(\sqrt{b/a})$, which implies $H(r'_y)<0$. 

Finally we prove that $S_+$ is open. We already know that $S_+\supset (0,\arcsin(\sqrt{2b/a}))$. Let then $y\in S_+\cap (\arcsin(\sqrt{2b/a}),\pi/2)$. For $z$ in a neighborhood of $y$, $u_z$ possesses a local minimum at $r_z'$ at which $H(r'_z)<0$. Since $(a/4)\sin^2(u(r))(\sin^2(u(r))-2b/a)\leq H(r)<H(r'_z)<0$, we get $0<\epsilon\leq u(r)\leq \arcsin(\sqrt{2b/a})-\epsilon$ for all $r>r'_z$ and some $\epsilon>0$, and therefore $z\in S_+$.
\end{proof}

Let now $v_y$ be the unique solution to
\begin{equation}
\begin{cases}
\dps L(v):=v''+\frac{d-1}{r}v'+F'(u_y)v=0\\
v(0)=1\\
v'(0)=0.
\end{cases}
\label{eq:linearized2}
\end{equation}
The main remark is that $v_y=\partial_y u_y$ is the variation of $u$ with respect to the initial condition $u_y(0)=y$, which implies the following result

\begin{lemma}
Assume that $y\in S_0$ and that $v_y(r),v_y'(r)\to-\ii$ when $r\to+\ii$. Then there exists $\epsilon>0$ such that $(y-\epsilon,y)\subset S_+$ and $(y,y+\epsilon)\subset S_-$.
\end{lemma}

\begin{proof}
This is~\cite[Lem. 3(b)]{McLeod-93} and the argument goes as follows.
Choose first $\alpha>0$ such that $F'\leq -b/2$ on $[0,\alpha)$, and then $\bar R$ such that $u(r)\leq \alpha$ for all $r\geq \bar R$. Finally, choose $R\geq \bar R$ such that $v_y(R)<0$ and $v'_y(R)<0$. For $z\in(y,y+\epsilon)$, we then have $0<u_z(R)<u_y(R)$ and $u_z'(R)<u_y'(R)<0$. The function $w:=u_z-u_y$ is negative at $R$ with $w'(R)<0$. If $z\in S_0$ or if $z\in S_+$, then $w$ must tend to $0$ or become positive at some point, and therefore it must have a first local (strict) minimum at some point $R'>R$, with $w(R)\geq w(r)\geq w(R')$ for all $R\leq r\leq R'$. From the equation~\eqref{eqNLSsinusradial} we can then write
$$w''(R')=F(u_y(R'))-F(u_z(R'))=-F'(\theta)w(R'),$$
for some $0<u_z(R')< \theta< u_y(R')\leq \alpha$. Here $u_z(R')>0$ because of our assumption that $z\in S_0\cup S_+$ and $u_y(R')\leq \alpha$ by choice of $\alpha$.  Now $F'(\theta)\leq -b/2<0$ and $w(R')<0$, which is a contradiction. The argument is the same for $z<y$.
\end{proof}

The lemma implies that any $y\in S_0$ for which $v_y,v'_y$ diverges to $-\ii$ must be an isolated point. Now, if we can prove that $v_y,v'_y\to-\ii$ for all $y\in S_0$ then we would clearly be done. Indeed, we know that $S_+$ and $S_-$ are open and they can only be separated by points in $S_0$. But the lemma says that points in $S_0$ can only serve as a transition between $S_+$ below and $S_-$ above. Therefore, there can be only one such transition, and we conclude that $S_0$ is reduced to one point. So our goal will be to prove that all the points in $S_0$ have $v_y,v'_y\to-\ii$.

Our argument will be based on the Wronskian identity
\begin{equation}
\big(r^{d-1}(v_yf'-fv_y')\big)'=r^{d-1}v_yL(f)
\label{eq:Wronskian}
\end{equation}
for various functions $f$'s. A simple calculation shows that
\begin{equation}
L(u_y)=u_yF'(u_y)-F(u_y),
\label{eq:L_u}
\end{equation}
\begin{equation}
L(ru_y')=-2F(u_y),
\label{eq:L_ru}
\end{equation}
and
\begin{equation}
L(u_y')=\frac{d-1}{r^2}u_y'.
\label{eq:L_u'}
\end{equation}
These three test functions correspond respectively to variations of $u_y$ using multiplication by a constant, dilations and translations.

\begin{lemma}\label{lem:v_has_one_root}
For every $y\in S_0$, the function $v_y$ vanishes exactly once.
\end{lemma}

\begin{proof}
For simplicity we denote again $u=u_y$ and $v=v_y$. Assume on the contrary that $v(r)>0$ for all $r\geq 0$ (if $v$ does not vanish it must be strictly positive since it cannot have double zeroes). Using~\eqref{eq:Wronskian} with $f=u'$, we find
$$\big(r^{d-1}(vu''-u'v')\big)'=(d-1)r^{d-3}v(r)u'(r)<0$$
and, therefore, $r^{d-1}(vu''-u'v')=r^{d-1}v^2(u'/v)'$ is decreasing and vanishes at $r=0$, hence $(u'/v)'<0$. Since $u'(0)/v(0)=0$, we conclude that $u'/v\leq -\epsilon$ for $r\geq1$ and thus $0\leq v\leq -u'/\epsilon$. As we have said $r^{d-1}(vu''-u'v')$ vanishes at $r=0$ and it is decreasing, hence $r^{d-1}(vu''-u'v')\leq -\epsilon$ for $r\geq1$. However $r^{d-1}|vu''|\leq C r^{d-1}|u'(r)|\,|u''(r)|$ decays exponentially at infinity and hence $r^{d-1}u'v'\geq \epsilon/2$ for $r$ large enough. Using~\eqref{eq:exponential_decay}, this proves that $-\sqrt{b}v'\geq C e^{\sqrt{b}r}r^{-(d-1)/2}$. Therefore $v'$ diverges to $-\ii$ exponentially at infinity, which contradicts the assumption that $v>0$. 

Next, the proof that $v$ can only vanish once is the same as in~\cite[p. 357--358]{Tao-06}. Indeed, start with $z=\arcsin(\sqrt{b/a})$ at which the solution $u_{z}$ is stationary. The function $u_{y}-u_{z}=u_y-\arcsin(\sqrt{b/a})$ vanishes exactly once since $u_{y}$ decreases from $y>\arcsin(\sqrt{b/a})=z$ to 0. Taking $z\to y$ and using that $u_y-u_{z}$ cannot have double zeroes gives that $v$ can vanish at most once.
\end{proof}

We are now able to show that $v$ and $v'$ diverge to $-\ii$. 

\begin{lemma}
Let $y\in S_0$. Then $v_y(r)$ and $v'_y(r)$ diverge to $-\ii$ as $r\to\ii$.
\end{lemma}

\begin{proof}
For simplicity we denote again $u=u_y$ and $v=v_y$. 
Let $r_*$ be the unique root of $v$, at which we must have $v'(r_*)<0$. Let now $c:=-u(r_*)/(r_*u'(r_*))>0$, which is chosen such that $f:=u+cru'$ vanishes at the zero $r_*$ of $v$. Recall that $u'(r)<0$ and $u(r)>0$ for all $r>0$, by Lemma~\ref{lem:prop_S_-}. Then we have from~\eqref{eq:Wronskian}
\begin{equation}
\big(r^{d-1}(f'v-v'f)\big)'=r^{d-1}v\Big(uF'(u)-(1+2c)F(u)\Big).
\label{eq:calcul_Wronskian}
\end{equation}
Next we remark that the function $r^{d-1}(f'v-v'f)$ vanishes both at $r=0$ and at $r=r_*$. Therefore, its derivative must vanish at least once on $(0,r_*)$, that is, $uF'(u)-(1+2c)F(u)$ vanishes before $r_*$. Since $u$ is strictly decreasing, and since $y\mapsto yF'(y)-(1+2c)F(y)$ vanishes only once by Lemma~\ref{lem:prop_F}, we conclude that $\big(r^{d-1}(f'v-v'f)\big)'$ is negative for $r>r_*$, hence $r^{d-1}(f'v-v'f)$ is strictly decreasing after $r_*$. In particular, 
$$r^{d-1}(v'f-vf')\geq \epsilon>0,\qquad \forall r>2r_*.$$
Since $f=u+cru'$ and $f'$ go to 0 exponentially at infinity, we conclude that $(v,v')$ must diverge. More precisely, we have for $r$ large enough
$$f=u(1+cru'/u)\underset{r\to\ii}{\sim}-C\sqrt{b}\,r^{(3-d)/2}e^{-\sqrt{b}r}$$
since $u'/u\to -\sqrt{b}$ and by~\eqref{eq:exponential_decay}. Hence
$$\left(\frac{v}{f}\right)'\geq \frac{\epsilon}{r^{d-1}f^2}\geq Cr^{-2}e^{2r\sqrt{b}}$$
and after integrating we get $v\leq -Ce^{r(\sqrt{b}-\epsilon)}$. As a consequence, $v$ diverge to $-\ii$ exponentially.

Finally using that $(r^{d-1}v')'=-r^{d-1}F'(u)v\leq (b/2)r^{d-1}v$ for large $r$ (since $F'\to- b$), we conclude that $v'$ diverges to $-\ii$ exponentially as well.
\end{proof}

As we have explained, the fact that all the points $y\in S_0$ are non-degenerate with $v_y,v'_y\to-\ii$ implies uniqueness, and concludes the proof of Theorem~\ref{thm:radial}.
\end{proof}

\bigskip

\begin{proof}[Proof of Lemma~\ref{lem:prop_F}]
Let $P(\xi)=a\xi^3-b\xi$ be the NLS polynomial, which is such that $F(x)=\cos(x)\,P(\sin(x))$. We have
\begin{equation}
\frac{\xi P'(\xi)}{P(\xi)}=3+\frac{2b}{a\xi^2-b}
\label{eq:calcul_P}
\end{equation}
which is positive decreasing on $(\sqrt{b/a},1)$. Noticing that
$$\frac{xF'(x)}{F(x)}=x\frac{\cos(x)}{\sin(x)}\left(\frac{\sin(x)\,P'(\sin(x))}{P(\sin(x))}\right)-\frac{x\sin(x)}{\cos(x)},$$
we find
\begin{align*}
\left(\frac{xF'(x)}{F(x)}\right)'&=\frac{\sin(2x)-2x}{2\sin^2(x)}\left(\frac{\sin(x)\,P'(\sin(x))}{P(\sin(x))}\right)\\
&\qquad\qquad+x\frac{\cos(x)}{\sin(x)}\left(\frac{\sin(x)\,P'(\sin(x))}{P(\sin(x))}\right)'-\frac{\sin(2x)+2x}{2\cos^2(x)}\\
&=\frac{\sin(2x)-2x}{2\sin^2(x)}\left(3+\frac{2b}{a\sin^2(x)-b}\right)\\
&\qquad\qquad-4abx\frac{\cos^2(x)}{(a\sin^2(x)-b)^2}-\frac{\sin(2x)+2x}{2\cos^2(x)}.
\end{align*}
This is negative for $\arcsin(\sqrt{b/a})<x<\pi/2$.

Let now $\lambda>1$, and consider the function $I$ in~\eqref{eq:def_I}. Note that $F'(0)=-b$, and hence $xF'(x)-\lambda F(x)=(\lambda-1)bx+o(x)$ is positive for small $x>0$. On the other hand, 
$$(\pi/2)F'(\pi/2)-F(\pi/2)=\pi(b-a)/2<0,$$
hence $I$ must vanish at least once on the interval $(0,\pi/2)$. Next we remark that 
\begin{align*}
I(x)&=a\left(x\cos(2x)-\frac{\lambda}{2}\sin(2x)\right)\left(\sin^2(x)-\frac{b}{a}\right)+\frac{a}{2}x\sin^2(2x)\nn\\
&=\frac{a}2\sin(2x)\left[ \left(\frac{2x\cos(2x)}{\sin(2x)}-\lambda\right)\left(\sin^2(x)-\frac{b}{a}\right)+x\sin(2x)\right].
\label{eq:calcul_vanish_once}
\end{align*}
Note that, for $0<x<\pi/2$, $\sin(2x)>0$  and
$$\frac{2x\cos(2x)}{\sin(2x)}-\lambda\leq 1-\lambda<0.$$
From this we conclude that $I(x)>0$ when $0<x\leq\arcsin(\sqrt{b/a}),$ hence $I$ can only vanish on $(\arcsin(\sqrt{b/a}),\pi/2)$. On this interval $xF'(x)/F(x)$ is strictly decreasing, as we have shown before, hence $I$ can only have one root.
\end{proof}

\subsection{Non-degeneracy in $L^2(\R^d)$}\label{sec:nondegeneracy_L2}

The linearized operators at our solution $\phi=\sin(Q)$ are defined by
\begin{equation}
L_1(\eta)=-\nabla\cdot \left(\frac{\nabla\eta}{1-\phi^2}\right)+\bigg\{-2\nabla\cdot\left(\frac{\phi\nabla\phi}{(1-\phi^2)^2}\right)+4\frac{\phi^2(\phi')^2}{(1-\phi^2)^3}+\frac{(\phi')^2}{(1-\phi^2)^2}-3a\phi^2+b\bigg\}\eta
\label{eq:def_L_1}
\end{equation}
and 
\begin{equation}
L_2(\eta)=-\nabla\cdot \left(\frac{\nabla\eta}{1-\phi^2}\right)+\bigg\{\frac{(\phi')^2}{(1-\phi^2)^2}-a\phi^2+b\bigg\}\eta.
\label{eq:def_L_2}
\end{equation}
More precisely, the linearized operator is $\cL(\eta_1+i\eta_2)=L_1\eta_1+iL_2\eta_2$. The operator $L_1$ describes variations with respect to $\phi$ for real functions, whereas $L_2$ is related to the invariance of our problem under multiplication by a phase factor. It is easy to verify that both $L_1$ and $L_2$ are self-adjoint operators on $L^2(\R^d)$, with domain $H^2(\R^d)$ and form domain $H^1(\R^d)$. The main result of this section is 

\begin{theorem}[Non-degeneracy of the unique ground state $\phi$]\label{thm:non-degenerate}
In $L^2(\R^d)$, we have $\ker(L_1)={\rm span}(\partial_{x_1}\phi,...,\partial_{x_d}\phi)$ and $\ker(L_2)={\rm span}(\phi)$.
\end{theorem}

\begin{proof}
The operators $L_1$ and $L_2$ both satisfy the Perron-Frobenius property that their first eigenvalue, when it exists, is necessarily non-degenerate with a positive eigenfunction. This follows for instance from the fact that $\pscal{\eta,L_{1/2}\eta}\geq \pscal{|\eta|,L_{1/2}|\eta|}$ and from Harnack's inequality~\cite[Sec. 6.4, Thm 5]{Evans} which gives the strict positivity of eigenfunctions. Since $L_2\phi=0$ and $\phi$ is positive, we deduce that it must be the first eigenfunction of $L_2$, and that it is non-degenerate. Thus $\ker(L_2)={\rm span}(\phi)$. Next, in dimension $d=1$, we know that $\partial_x \varphi\in \ker(L_1)$ and $\partial_x \varphi$ has a constant sign. Hence, $0$ is the first eigenvalue of $L_1$ and it is non-degenerate which implies $\ker(L_1)={\rm span}(\partial_x \varphi)$.

The argument for $L_1$ in dimension $d\ge 2$ is slightly more complicated. A lengthy but straightforward computation shows that
$$L_1(\eta)=-\frac{\Delta v+F'(Q)v}{\cos(Q)},\qquad \text{with } v=\frac{\eta}{\cos(Q)}.$$
Since $0<Q\leq Q(0)<\pi/2$, the multiplier $\cos(Q)$ is bounded away from $0$ and we deduce that $v\in L^2(\R^d)$ if and only if $\eta\in L^2(\R^d)$. Hence $\eta\in \ker(L_1)$ if and only if $v=\eta/ \cos(Q)\in \ker(\Delta +F'(Q))$. The argument is now classical. The operator $-\Delta -F'(Q)$ commutes with space rotations and it may be written as a direct sum 
$$-\Delta-F'(Q)=\bigoplus_{\ell\geq0}A^{(\ell)}\otimes \1$$
corresponding to the decomposition
$$L^2(\R^d)=\bigoplus_{\ell\geq0}L^2(\R_+,r^{d-1}\,dr)\otimes \cK_\ell$$ 
with $\cK_\ell=\ker\big(\Delta_{|S^{d-1}}+\ell(\ell+d-2)\big)$ the $\ell$th  eigenspace of the Laplace-Beltrami operator on the sphere $S^{d-1}$. In dimension $d=3$, $\cK_\ell={\rm span}\{Y_m^{(\ell)},m=-\ell,...,\ell\}$ where $Y_m^{(\ell)}$ are the usual spherical harmonics. The formula for $A^{(\ell)}$ is 
$$A^{(\ell)}v:=-v''-\frac{(d-1)}{r}v'+\frac{\ell(\ell+d-2)}{r^2}v-F'(Q(r))v $$
with an appropriate boundary condition at $r=0$ (Neumann for $\ell=0$ and Dirichlet for $\ell\geq1$). Each $A^{(\ell)}$ has the Perron-Frobenius property. Since $Q'\in \ker(A^{(1)})$ and $Q'$ has a constant sign, we conclude that $0$ is the first eigenvalue of $A^{(1)}$ and it is non-degenerate, thus $\ker(A^{(1)})={\rm span}(Q')$. Next, for $\ell\geq2$, we simply use that $A^{(\ell)}>A^{(1)}$ in the sense of quadratic forms, which shows that the first eigenvalue of $A^{(\ell)}$ must be positive and hence $\ker(A^{(\ell)})=\{0\}$. Finally, for $\ell=0$, the operator $A^{(0)}$ was studied in Theorem~\ref{thm:radial}, where we proved that the unique solution to $A^{(0)}v=0$ with $v'(0)=0$ diverges exponentially at infinity, hence cannot be in $L^2(\R_+,r^{d-1}\,dr)$. We have therefore shown that 
$$\ker(\Delta+F'(Q))={\rm span}\{\partial_{x_1}Q,...,\partial_{x_d}Q\}.$$
Since $\partial_{x_k}\phi=\partial_{x_k}\sin(Q)=\cos(Q)\partial_{x_k}Q$, this says that 
$$\ker(L_1)={\rm span}\{\partial_{x_1}\phi,...,\partial_{x_d}\phi\}$$
which concludes our proof of Theorem~\ref{thm:non-degenerate}.
\end{proof}

\section{Proof of Theorem~\ref{thm:Dirac}}\label{sec:relativistic_limit}

We want to prove the existence of a branch of solutions to the Dirac equation
\begin{equation}
\begin{cases}
-i\alp\cdot\nabla \Psi+\beta(m+S)\Psi+V\Psi=(m-\mu)\Psi,\\
(-\Delta+m_\sigma^2)S=-g_\sigma^2\Psi^*\beta\Psi,\\
(-\Delta+m_\omega^2)V=g_\omega^2|\Psi|^2,\\
\end{cases}
\label{eq:Dirac2}
\end{equation}
which can be rewritten for $\Psi=(\psi,\zeta)$ as
\begin{equation}
\label{eqnrlphichi3d}
\begin{cases}
-i\bm{\sigma}\cdot\nabla \zeta +(S+V+\mu)\psi=0,\\
-i\bm{\sigma}\cdot\nabla \psi =(2m-\mu+S-V)\zeta,\\
(-\Delta+m_\sigma^2)S=-g_\sigma^2(|\psi|^2-|\zeta|^2),\\
(-\Delta+m_\omega^2)V=g_\omega^2(|\psi|^2+|\zeta|^2).
\end{cases}
\end{equation}
Here the parameters are chosen as 
\begin{equation}
m_\sigma^2=Cm^2,\quad m^2_\omega-m_\sigma^2=D,\quad \left(\frac{g_\sigma}{m_{\sigma}}\right)^2=\theta m,\quad \left(\frac{g_\sigma}{m_{\sigma}}\right)^2-\left(\frac{g_\omega}{m_{\omega}}\right)^2=\lambda
\label{eq:NR-limit-parameters2}
\end{equation}
with $C,D,\theta,\lambda,\mu>0$ fixed such that $\lambda>2\theta\mu$. It will be convenient to introduce 
the new fields
\begin{equation}
	\label{changepot}
	\tilde W_+=\frac{S+V}{2}\ \text{and}\ \tilde W_-=\frac{S-V}{2}
\end{equation}
Then, imposing the special form
\begin{equation}
\psi(x)=\tilde \varphi(|x|)\vct{1}{0},\qquad \zeta(x)=-i\tilde \chi(|x|)\;\bm\sigma\cdot\frac{x}{|x|}\vct{1}{0},
\label{eq:form_Dirac2}
\end{equation}
with real-valued functions $\tilde \phi$ and $\tilde \zeta$, and using \eqref{eq:NR-limit-parameters2} and \eqref{changepot}, we obtain the following system
\begin{equation}
	\label{systemW+W-3d}
	\left\{
	\begin{aligned}
		&\tilde \phi' - (2m + 2\tilde W_--\mu)\tilde \chi=0\\
		&\tilde \chi'+\frac{2}{r}\tilde \chi-(2 \tilde W_++\mu) \tilde \phi=0\\
		&\tilde W_+ = \frac{1}{2}\left(\frac{1}{m^2_{\sigma}}+\frac{1}{m_{\omega}^2}\right)\Delta \tilde W_+ + \frac{1}{2}\left(\frac{1}{m^2_{\sigma}}-\frac{1}{m_{\omega}^2}\right)\Delta \tilde W_--\frac{\lambda}{2}(\tilde \phi^2+\tilde\chi^2)+\theta m\tilde \chi^2\\
		&\tilde W_- = \frac{1}{2}\left(\frac{1}{m^2_{\sigma}}-\frac{1}{m_{\omega}^2}\right)\Delta \tilde W_+ + \frac{1}{2}\left(\frac{1}{m^2_{\sigma}}+\frac{1}{m_{\omega}^2}\right)\Delta \tilde W_- +\frac{\lambda}{2}(\tilde \phi^2+\tilde \chi^2)-\theta m\tilde \phi^2
	\end{aligned}
	\right.
\end{equation}
which is equivalent to \eqref{eq:Dirac2} for functions of the above form~\eqref{eq:form_Dirac2}.

Next, we consider the following rescaling 
\begin{equation}
	\label{rescaling}
	\begin{aligned}
&\tilde \varphi(x)=\frac{1}{\sqrt{\theta}}\varphi (\sqrt{m}x),\qquad\tilde \chi(x)=\frac{1}{2\sqrt{\theta}}\frac{1}{\sqrt{m}} \chi(\sqrt{m}x),\\
&\tilde W_+(x)=W_+(\sqrt{m}x),\qquad \tilde W_-(x)=m W_-(\sqrt{m}x),
	\end{aligned}
\end{equation}
and we find
\begin{equation*}
	\left\{
	\begin{aligned}
		&  \varphi'-\left(1 +  W_-- \frac{\mu}{2 m}\right)  \chi=0\\
		&\chi'+\frac{2}{r}\chi-\left(4  W_+ + 2\mu\right) \varphi=0\\
		& W_+ = \frac{2+D/(Cm^2)}{2m(C+D/m^2)}\Delta  W_+ + \frac{D}{2C(Cm^2+D)}\Delta  W_--\frac{\lambda}{2}\left(\frac{\varphi^2}{\theta}+\frac{\chi^2}{4\theta m}\right)+\frac{\chi^2}{4}\\
		& W_- = \frac{D}{2Cm^2(Cm^2+D)} \Delta  W_+ +\frac{2+D/(Cm^2)}{2m(C+D/m^2)} \Delta  W_- +\frac{\lambda}{2 m}\left(\frac{ \varphi^2}{\theta}+\frac{\chi^2}{4\theta m}\right)- \varphi^2
	\end{aligned}
	\right.
\end{equation*}
Finally, denoting $\epsilon=1/m$ the perturbative parameter and recalling that 
$$a=2\lambda/\theta,\qquad b=2\mu,$$
we obtain
\begin{equation}
	\label{perturbedsystem3d}
	\left\{
	\begin{aligned}
		&\phi' -\left(1 +   W_-- \epsilon\frac{b}{4}\right) \chi=0\\
		&\chi'+\frac{2}{r}\chi-\left(4  W_++b\right) \varphi=0\\
		&\big(-\epsilon\mathcal{R}(\epsilon)\Delta+\mathds{1}_2\big)\vct{ W_+}{ W_-}+\mathcal{F}(\phi,\chi)+\mathcal{H}(\epsilon,\phi,\chi)=0
	\end{aligned}
	\right.
\end{equation}
with 
$$\cR(\epsilon)=\frac{1}{2(C+D\epsilon^2)}\left(\begin{array}{cc}2+\epsilon^2 D/C & \epsilon D/C\\
	\epsilon^3D/C & 2+\epsilon^2D/C 
	\end{array}\right),$$
$$\mathcal{F(\phi,\chi)=\vct{a\phi^2/4-\chi^2/4}{\phi^2}},\qquad \mathcal{H}(\epsilon,\phi,\chi)=\epsilon\frac{a}{4}\vct{\chi^2/4}{-\phi^2-\epsilon\chi^2/4}.$$
When $\epsilon=0$, we obtain the system of equations
\begin{equation}
\begin{cases}
\varphi'=\chi(1-\varphi^2)\\
\chi'+\frac{2}{r}\chi=\varphi(\chi^2-a \varphi^2+b)
\end{cases}
\label{eq:NLS_system_radial}
\end{equation}
which is equivalent to~\eqref{eqNLSsinusradial} with $\phi=\sin(u)$ and was studied in~\cite{EstRot-12,EstRot-13}.

We introduce the map
$\cK:\RR\times \cH^2_{\rm rad} \times (H^{2}_{\mathrm{rad}})^2\longrightarrow \cH^1_{\rm rad}\times (H^{2}_{\mathrm{rad}})^2$ defined by
\begin{multline}
	\label{opkrad}
	\cK(\epsilon,\varphi,\chi,W_+,W_-)=\\
	\begin{pmatrix}
		\varphi' - (1 + W_-- \epsilon b/4)\chi\\
		\chi'+2\chi/r -(4 W_++b) \varphi\\
		\begin{pmatrix}W_+\\ W_-\end{pmatrix}+\frac{\dps1}{\dps\epsilon\cR(\epsilon)(-\Delta)+\1_2}\Big(\mathcal{F}(\varphi,\chi)+\mathcal{H}(\epsilon,\varphi,\chi)\Big)
	\end{pmatrix}.
\end{multline}
Here the spaces
$$\cH^k_{\rm rad}:=\left\{(\phi,\chi)\ :\ \vct{\phi(|x|)\vct{1}{0}}{-i\chi(|x|)\,\bm\sigma\cdot\frac{x}{|x|}\vct{1}{0}}\in H^k(\R^3,\C^4)\right\}$$
are the projections of the usual Sobolev spaces $H^k(\R^d,\C^4)$ to the sector of minimal total angular momentum (they in particular contain a boundary condition at $r=0$), whereas $H^2_{\rm rad}$ is the usual projection of $H^2(\R^3,\R)$ to the subspace of radial functions. 

In what follows, we let $\X=\cH^2_{\rm rad} \times (H^{2}_{\mathrm{rad}})^2$, $\Y= \cH^1_{\rm rad}\times (H^{2}_{\mathrm{rad}})^2$ and $\Xi=(\phi,\chi,W_+,W_-)$. Solving the system \eqref{perturbedsystem3d} is equivalent to solving $\cK(\epsilon,\Xi)=0$. We construct a branch of solutions, by means of an implicit function-type argument. 
The first step is to prove that $\cK$ is a smooth operator from $\RR\times \X$ into $\Y$.


\begin{lem}
\label{Kradregularity}
For $\eta$ small enough, the operator $\cK:[0,\eta)\times \X \to \Y$ defined as in \eqref{opkrad} is continuous. Its derivative $\partial_\Xi\cK:[0,\eta)\times \X \to \Y$ is also continuous.
\end{lem}

Note that we do note prove the continuity of the derivative $\partial_\epsilon\cK$, which fails at $\epsilon=0$. Fortunately, the latter is not needed for the implicit function theorem (see e.g.\cite[Thm.~3.4.10]{KraPar-12}).

\begin{proof} 
The proof is tedious but elementary. It relies on the fact that $H^2(\R^3)$ is an algebra and that all the functions appearing in the definition of $\cK$ are polynomials in the unknowns $(\phi,\chi,W_+,W_-)$. Also, it uses that 
$$\vct{\phi}{\chi}\mapsto \vct{\phi'-\chi}{\chi'-2\chi/r-b\phi}$$
is an isomorphism from $\cH^2_{\rm rad}$ to $\cH^1_{\rm rad}$. Indeed, this map is related to the restriction of the Dirac operator
$$-i\alp\cdot\nabla+\frac{b+1}{2}\beta+\frac{b-1}{2}$$
to functions of the form~\eqref{eq:form_Dirac2}. Since $|b-1|/2<(b+1)/2$, the operator is an isomorphism from $H^2(\R^3)$ to $H^1(\R^3)$ and the same holds in the radial subspaces $\cH^2_{\rm rad}$ and $\cH^1_{\rm rad}$. Similarly, the operator
$$\Phi\mapsto \frac{\dps1}{\dps\epsilon\cR(\epsilon)(-\Delta)+\1_2}\Phi$$
is the Fourier multiplier with the matrix $(\epsilon\cR(\epsilon)|k|^2+\1_2)^{-1}$ and we claim that
\begin{equation}
\norm{\big(\epsilon\cR(\epsilon)|k|^2+\1_2\big)^{-1}}\leq 1
\label{eq:estim_R}
\end{equation}
for all $k\in\R^3$ and all $\epsilon\ge 0$. 
This estimate shows that the corresponding map is bounded on $H^2(\R^3)$, as needed. In order to prove~\eqref{eq:estim_R}, we recall that
$$\epsilon\cR(\epsilon)|k|^2+\1_2=\frac{\epsilon|k|^2}{2C(C+D\epsilon^2)}\begin{pmatrix}2C+\epsilon^2 D & \epsilon D\\
\epsilon^3D & 2C+\epsilon^2D\\ 
\end{pmatrix}+\1_2.$$
Changing $\epsilon|k|^2$ into $\epsilon|k|^2/(2C^2+2DC\epsilon^2)$, it suffices to show that 
x$$\norm{M_\epsilon(x)^{-1}}\leq\frac{1}{1+cx}$$
for all $x\geq0$, with
\begin{align*}
M_\epsilon(x)&:=x\begin{pmatrix}2C+\epsilon^2 D & \epsilon D\\
	\epsilon^3D & 2C+\epsilon^2D\\ 
	\end{pmatrix}+\1_2
\end{align*}
The matrix $M_\epsilon(x)$ has two real positive eigenvalues $\xi_-=1+2Cx$ and $\xi_+=1+ 2x(C+\epsilon^2D)+1$ with $\xi_-<\xi_+$. Hence 
\begin{align*}
\norm{M_\epsilon(x)^{-1}}^{-1}&=1+2Cx.
\end{align*}
As a consequence, $\|M_\epsilon(x)^{-1}\|= (1+2Cx)^{-1}\leq 1$, for all $k\in\R^3$ and for all $\epsilon\ge 0$. 

Finally, the fact that $\partial_\Xi\cK:[0,\eta)\times \X \to \Y$ is also continuous can be proved with the same arguments.
\end{proof}


Next, we consider the linearization $\cL=\partial_\Xi\cK(0,\Xi_0)$ of the operator $\cK$ at our non-relativistic solution $\Xi_0=(\varphi,\chi,W_+,W_-)\in X$ with
$$\chi=\phi'/(1-\phi^2),\qquad W_+=-\frac{a}4\phi^2+\frac{1}4\chi^2,\qquad W_-=-\phi^2,$$
which is defined by
\begin{equation}
	\label{linearizationoprad}
	\cL(f,g,h_+,h_-)=
	\left(
	\begin{array}{c}
		f' -(1+W_{-})g-\chi h_-\\
		g'+\frac{2}{r}g  - 4 W_{+}f -4\varphi h_+ -b f \\
		h_+ + \frac{a}{2}\varphi f -\frac{1}{2}\chi g\\
		h_- + 2 \varphi f
	\end{array}
	\right).
\end{equation}

\begin{lem}
	\label{Kliniso3d} The operator $\cL:\X\to \Y$ defined as in \eqref{linearizationoprad} is an isomorphism.
\end{lem}

\begin{proof}
	First we prove that $\cL$ is a one to one operator. Let $(f,g,h_+,h_-)\in \X$ a nontrivial solution to $\cL(f,g,h_+,h_-)=0$. Then, since $(\cR(0)(-\Delta)+\1_2)^{-1}$ is bounded, 
	$$\vct{h_+}{h_-}=\vct{-\frac{a}{2}\varphi f +\frac{1}{2}\chi g}{- 2 \varphi f}$$
	and $(f,g)$ solves
	\begin{equation}
		\left\{
		\begin{aligned}
			&f' -(1-\varphi^2)g+2\varphi\chi f=0\\
			& g'+\frac{2}{r}g - (\chi^2-3a\varphi^2+b)f -2\varphi \chi g=0 \\
		\end{aligned}
		\right..
	\end{equation}
	A calculation shows that the radial function $f$ solves $L_1f=0$ where $L_1$ is the linearized operator defined in~\eqref{eq:def_L_1}. Since the restriction of $L_1$ to radial functions is invertible by Theorems~\ref{thm:radial} and~\ref{thm:non-degenerate}, we conclude that $(f,g)=(0,0)$ and $\cL$ is one-to-one. 

	Next, we observe that $\cL$ can be written as the sum of two linear operators
	\begin{equation*}
		\cL(f,g,h_+,h_-)=\left(
		\begin{array}{c}
			f' -g\\
			g'+\frac{2}{r}g  -b f \\
			h_+ \\
			h_- 
		\end{array}
		\right)
		+
		\left(
	\begin{array}{c}
		-W_{-}g-\chi h_-\\
		- 4 W_{+}f -4\varphi h_+ \\
		\frac{a}{2}\varphi f -\frac{1}{2}\chi g\\
		2 \varphi f
	\end{array}
	\right):=\cL_1+\cL_2.
	\end{equation*}
As we have already said before, the upper part of the operator $\cL_1$ is a restriction of the Dirac operator $-i\alp\cdot\nabla+(b+1)\beta/2+(b-1)/2$ and it is an isomorphism from $\cH^2_{\rm rad}$ to $\cH^1_{\rm rad}$, by definition of these spaces. On the other hand, $\cL_2$ is compact. Therefore $\cL$ is a one-to-one operator that can be written as a sum of an isomorphism and a compact perturbation and it is then an isomorphism. 
\end{proof}

As a conclusion, we can apply the implicit function theorem to find that there exists $\delta>0$ and a function $\Xi\in \mathcal C([0,\delta)\times X)$ such that 
$$\Xi(0)=\left(\phi\,,\,\frac{\phi'}{1-\phi^2}\,,\,-\frac{a}4\phi^2+\frac{1}4\chi^2\,,\,-\phi^2\right)$$ 
and $\cK\big(\epsilon,\Xi(\epsilon)\big)=0$ for $0\le \epsilon <\delta$. This concludes the proof of Theorem~\ref{thm:Dirac}.\qed

\bigskip


\end{document}